\documentclass[12pt]{amsart}
\usepackage[dvips]{graphicx}
\usepackage{amssymb,amsmath,amsthm}

\newcommand\figseven{.}

\def\frk{\mathfrak}               % font for "Fraktur"

\def\Phi{{\frk N}}
\def\tb{{\bf t}}

\def\xb{{\bf x}}

\newtheorem{Theorem}{Theorem}[section]
\newtheorem{Lemma}[Theorem]{Lemma}

\newtheorem{Proposition}[Theorem]{Proposition}

\theoremstyle{definition}
\newtheorem{Remark}[Theorem]{Remark}

\newtheorem{Example}[Theorem]{Example}

\begin{document}

\title{Many toric ideals generated by quadratic binomials possess no
 quadratic Gr\"obner bases}
\author{Takayuki Hibi, Kenta Nishiyama, Hidefumi Ohsugi and Akihiro Shikama}

\address{Takayuki Hibi,
Department of Pure and Applied Mathematics,
Graduate School of Information Science and Technology,
Osaka University,
Toyonaka, Osaka 560-0043, Japan}
\email{hibi@math.sci.osaka-u.ac.jp}

\address{Kenta Nishiyama,
School of Management and Information,
University of Shizuoka,
Suruga-ku, Shizuoka 422-8526, Japan}
\email{k-nishiyama@u-shizuoka-ken.ac.jp}

\address{Hidefumi Ohsugi,
Department of Mathematics,
College of Science,
Rikkyo University,
Toshima-ku, Tokyo 171-8501, Japan} 
\email{ohsugi@rikkyo.ac.jp}

\address{Akihiro Shikama,
Department of Pure and Applied Mathematics,
Graduate School of Information Science and Technology,
Osaka University,
Toyonaka, Osaka 560-0043, Japan}
\email{a-shikama@cr.math.sci.osaka-u.ac.jp}

\begin{abstract} 
Let $G$ be a finite connected simple graph and $I_{G}$ the toric ideal
of the edge ring $K[G]$ of $G$.
In the present paper, we study finite graphs $G$ 
with the property that 
$I_{G}$ is generated by quadratic binomials and $I_{G}$ 
possesses no quadratic Gr\"obner basis.
First, we give a nontrivial infinite series of finite graphs
with the above property.
Second, we implement a combinatorial characterization for
$I_{G}$ to be generated by quadratic binomials and,
by means of the computer search, we
classify the finite graphs $G$ with the above property,
up to $8$ vertices.
\end{abstract}

\maketitle

\section*{Introduction}

Let $G$ be a finite connected simple graph on the vertex set
$[n] = \{ 1, 2, \ldots, n \}$ with $E(G) = \{ e_{1}, \ldots, e_{d} \}$ its edge set.  
(Recall that a finite graph is {\em simple} if it possesses no loop and no multiple edge.)
Let $K$ be a field and $K[\tb] = K[t_{1}, \ldots, t_{n}]$ the polynomial ring in $n$ variables
over $K$.  If $e = \{i, j\} \in E(G)$, then $\tb^{e}$ stands for the quadratic monomial
$t_{i}t_{j} \in K[\tb]$.  The {\em edge ring} (\cite{quadratic}) of $G$ is the subring 
$K[G] = K[\tb^{e_{1}}, \ldots, \tb^{e_{d}}]$ of $K[\tb]$.
Let $K[\xb] = K[x_{1}, \ldots, x_{d}]$ denote the polynomial ring in $d$ variables over $K$
with each $\deg x_{i} = 1$
and define the surjective ring homomorphism $\pi : K[\xb] \to K[G]$ by setting
$\pi(x_{i}) = \tb^{e_{i}}$ for each $1 \leq i \leq d$.
The {\em toric ideal} $I_{G}$ of $G$ is the kernel $\pi$.
It is known \cite[Corollary 4.3]{Stu} that $I_{G}$ is generated by those binomials
$u - v$, where $u$ and $v$ are monomials of $K[\xb]$ with
$\deg u = \deg v$, such that $\pi(u) = \pi(v)$.

The distinguished properties on $K[G]$ and $I_{G}$ 
in which commutative algebraists are especially interested are as follows: 
\begin{enumerate}
\item[(i)]
$I_{G}$ is generated by quadratic binomials;
\item[(ii)]
$K[G]$ is Koszul;
\item[(iii)]
$I_{G}$ possesses a quadratic Gr\"obner basis, i.e., a Gr\"obner basis 
consisting of quadratic binomials.
\end{enumerate}
The hierarchy
(iii) $\Rightarrow$ (ii) $\Rightarrow$ (i)
is true.  However, 
%the converse of each of the implications
(i) $\Rightarrow$ (ii)  is false.
We refer the reader to \cite{quadratic} for the quick information 
together with basic literature on these properties.
A Koszul toric ring whose toric ideal possesses no quadratic Gr\"obner basis
is given in \cite[Example 2.2]{quadratic}.
Moreover, consult, e.g., to \cite[Chapter 2]{JuTa} for fundamental materials on Gr\"obner bases.

We study finite connected simple graphs $G$ 
satisfying the following condition:
\begin{enumerate}
\item[$(*)$]
$I_{G}$ is generated by quadratic binomials and $I_{G}$ 
possesses no quadratic Gr\"obner basis.
\end{enumerate}
We say that a finite connected simple graphs $G$ is {\em $(*)$-minimal}
if $G$ satisfies the condition $(*)$ and if no induced subgraph
$H$ $(\neq G)$ satisfies the condition $(*)$.
A $(*)$-minimal graph is given in \cite[Example 2.1]{quadratic}. 

In the present paper, after summarizing known results on $I_{G}$ in Section $1$,  
a nontrivial infinite series of $(*)$-minimal finite graphs is given
in Section $2$.
In Section $3$, 
we implement a combinatorial characterization for
$I_{G}$ to be generated by quadratic binomials (\cite[Theorem 1.2]{quadratic}) and,
by means of the computer search, we
classify the finite graphs $G$ satisfying the condition $(*)$,
up to $8$ vertices. 

Finally, an outstanding problem is to find a finite graph $G$
for which $K[G]$ is Koszul, but $I_{G}$ possesses no quadratic Gr\"obner basis.
We do {\em not} know that 
there exists such an example in our infinite series of $(*)$-minimal finite graphs. 

\section{Known results on toric ideals of graphs}

In this section, we introduce graph theoretical terminology and
known results. % which are useful later.
Let $G$ be a connected graph with the vertex set $V(G) = [n] = \{1,2,\ldots,n\}$
and the edge set $E(G)$.
We assume that $G$ has no loops and no multiple edges.
A walk of length $q$ of $G$ connecting $v_1\in V(G)$ and $v_{q+1} \in V(G)$ is 
a finite sequence of the form
\begin{equation}
\label{walk}
\Gamma =
 (
\{v_1,v_2\},
\{v_2,v_3\},
\ldots,
\{v_{q},v_{q+1}\}
)
\end{equation}
with each $\{v_k,v_{k+1}\} \in E(G)$.
An {\em even} (resp.~{\em odd}) {\em walk} is a walk of even (resp.~odd) length.
A walk $\Gamma$ of the form (\ref{walk}) is called {\em closed} if $v_{q+1} =v_1$. 
A {\em cycle} is a closed walk
\begin{equation}
\label{defofcycle}
C = 
 (
\{v_1,v_2\},
\{v_2,v_3\},
\ldots,
\{v_{q},v_{1}\}
)
\end{equation}
with $q \geq 3$ and $v_i \neq v_j$ for all $1 \leq i < j \leq q$.
A {\em chord} of a cycle (\ref{defofcycle}) is an edge $e \in E(G)$ of the form
$e = \{v_i,v_j\}$ for some $1 \leq i < j \leq q$ with $e \notin E(C)$.
If a cycle (\ref{defofcycle}) is even, 
an {\em even-chord} (resp.~{\em odd-chord}) of (\ref{defofcycle})
is a chord $e = \{v_i,v_j\}$ with $1 \leq i < j \leq q$ 
such that $j-i$ is odd (resp.~even).
If $e = \{v_i,v_j\}$ and $e' = \{v_{i'},v_{j'}\}$ are chords of a cycle  (\ref{defofcycle})
with $1 \leq i < j \leq q$ and $1 \leq i' < j' \leq q$,
then we say that $e$ and $e'$ {\em cross} in $C$
if the following conditions are satisfied:
\begin{itemize}
\item[(i)]
Either $i < i' < j < j'$ or $i' < i < j' < j$; 
\item[(ii)]
Either $\{  \{v_i,v_{i'}\}, \{v_j, v_{j'}\}  \} \subset E(C)$ or 
$\{  \{v_i,v_{j'}\}, \{v_j, v_{i'}\}  \} \subset E(C)$.
\end{itemize}
A {\em minimal} cycle of $G$ is a cycle having no chords.
If $C_1$ and $C_2$ are cycles of $G$ having no common vertices,
then a {\em bridge} between $C_1$ and $C_2$ is an edge $\{i,j\}$ of $G$
with $i \in V(C_1)$ and $j \in V(C_2)$.

The toric ideal $I_G$ is generated by the binomials associated with
even closed walks.
Given an even closed walk
$
\Gamma = (e_{i_1},e_{i_2},\ldots,e_{i_{2q}})
$
of $G$, we write $f_\Gamma$ for the binomial
$$
f_\Gamma = 
\prod_{k=1}^q x_{i_{2k-1}} -
\prod_{k=1}^q x_{i_{2k}}
\in I_G.
$$
It is known (\cite[Proposition 3.1]{Vil}, \cite[Chapter 9]{Stu}
and \cite[Lemma 1.1]{quadratic})
that 

\begin{Proposition}
Let $G$ be a connected graph.
Then, $I_G$ is generated by all the binomials
$
f_\Gamma
$,
where 
$\Gamma$ is an even closed walk of $G$.
In particular, $I_G=(0)$ if and only if 
$G$ has at most one cycle and the cycle is odd.
\end{Proposition}

Note that, for a binomial $f \in I_G$,
$\deg (f) =2$ 
if and only if there exists an even cycle $C$ of $G$ of length 4
such that $f = f_C$.
On the other hand, a criterion for the existence of a quadratic binomial generators of  $I_G$
is given in \cite[Theorem 1.2]{quadratic}.

\begin{Proposition}
\label{quadcriterion}
Let $G$ be a finite connected graph.
Then, $I_G$ is generated by quadratic binomials\footnote{
Even if $I_G  =(0)$, we say that ``$I_G$ is generated by quadratic binomials"
and ``$I_G$ possesses a quadratic Gr\"obner basis."
} if and only if the following conditions are
satisfied:
\begin{itemize}
\item[(i)]
If $C$ is an even cycle of $G$ of length $\geq 6$, then either $C$ has an even-chord
or $C$ has three odd-chords $e$, $e'$ and $e''$ such that $e$ and $e'$ cross in $C$;
\item[(ii)]
If $C_1$ and $C_2$ are minimal odd cycles having exactly one common vertex,
then there exists an edge $\{i,j\} \notin E(C_1) \cup E(C_2)$ with $i \in V(C_1)$
and $j \in V(C_2)$;
\item[(iii)]
If $C_1$ and $C_2$ are minimal odd cycles having no common vertex,
then there exist at least two bridges between $C_1$ and $C_2$.
\end{itemize}
\end{Proposition}

If $G$ is bipartite, then the following is shown in
\cite{Koszul}:

\begin{Proposition}
\label{Koszulbipartite}
Let $G$ be a bipartite graph.
Then the following conditions are equivalent:
\begin{itemize}
\item[(i)]
Every cycle of $G$ of length $\geq 6$ has a chord;
\item[(ii)]
$I_G$ possesses a quadratic Gr\"obner basis;
\item[(iii)]
$K[G]$ is Koszul;
\item[(iv)]
$I_G$ is generated by quadratic binomials.
\end{itemize}
\end{Proposition}

If $G$ is not bipartite, then the conditions (iii) and (iv) are not equivalent.

\begin{Example} (\cite[Example 2.1]{quadratic})
\label{syarin}
Let $G$ be the graph in Figure \ref{syarin5}.
Then, $I_G$ is generated by quadratic binomials.
On the other hand, $K[G]$ is not Koszul and hence
$I_G$ has no quadratic Gr\"obner bases.
\begin{figure}%[h]
  \begin{center}
    \scalebox{0.35}{\includegraphics{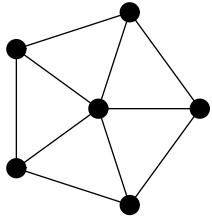}}
  \end{center}
  \caption{Wheel with 6 vertices.}
  \label{syarin5}
\end{figure} 

\end{Example}

If a graph $G'$ on the vertex set $V(G') \subset V(G)$ satisfies 
$E(G') = \{ \{i,j\} \in E(G) \ | \ i,j \in V(G') \}$, then 
$G'$ is called an {\em induced subgraph} of $G$.
The following proposition is a fundamental and important fact on
the toric ideals of graphs.

\begin{Proposition}[\cite{cpure}]
\label{inducedsubgraph}
Let $G'$ be an induced subgraph of a graph $G$.
Then, $K[G']$ is a combinatorial pure subring of $K[G]$.
In particular, 
\begin{itemize}
\item[(i)]
If $I_G$ possesses a quadratic Gr\"obner basis, then so does $I_{G'}$.
\item[(ii)]
If $K[G]$ is Koszul, then so is $K[G']$;
\item[(iii)]
If $I_G$ is generated by quadratic binomials, then so is $I_{G'}$.
\end{itemize}
\end{Proposition}

\section{Toric ideals of the suspension of graphs}

In this section, we study 
the existence of quadratic Gr\"obner bases of 
toric ideals of the suspension of graphs.

Let $G$ be a graph with the vertex set $V(G) = [n] = \{1,2,\ldots,n\}$
and the edge set $E(G)$.
The suspension of the graph $G$ is the new graph $\widehat{G}$
whose vertex set is $[n+1] = V(G) \cup \{n+1\}$
and whose edge set is $E(G) \cup \{ \{i,n+1\}  \ | \ i \in V(G)\}$.
Note that, any graph $G$ is an induced subgraph of its suspension $\widehat{G}$.
The {\em edge ideal} of $G$ is the monomial ideal $I(G)$ of 
$K[\tb]$ which is generated by $\{t_i t_j \ | \ \{i,j\} \in E(G)\}$.
See, e.g., \cite[Chapter 9]{JuTa}.
It is easy to see that the edge ring $K[\widehat{G}] \simeq K[{\bf x}]/ I_{\widehat{G}}$
of the suspension $\widehat{G}$ of $G$
is isomorphic to
the {\em Rees algebra}
$$
{\mathcal R}(I(G)) =
\bigoplus_{j=0}^\infty
I(G)^j s^j
=
K[t_1,\ldots,t_n, \{ t_i t_j s\}_{\{i,j\}\in E(G)}]
$$
of the edge ideal
$
I(G)
$
of $G$.

We now characterize graphs $G$ such that
$I_{\widehat{G}}$ is generated by quadratic binomials.
The {\em complementary} graph $\overline{G}$ of $G$
is the graph whose vertex set is $[n]$
and whose edges are the non-edges of $G$.
A graph $G$ is said to be {\em chordal} if any cycle of length $>3$
has a chord.
Moreover, a graph $G$ is said to be {\em co-chordal} if $\overline{G}$ is chordal.
A graph $G$ is called a $2K_2$-{\em free graph} if
it is connected and does not contain two
independent edges as an induced subgraph.
%It is easy to see that, f
For a connected graph $G$,
\begin{itemize}
\item
$G$ is $2K_2$-free $\Leftrightarrow $ 
any cycle of $\overline{G}$ of length $4$ has a chord in $\overline{G}$;
\item
$G$ is co-chordal $\Rightarrow $ $G$ is $2K_2$-free,
\end{itemize}
hold in general.
Moreover, it is known (e.g., \cite{graphpaper}) that

\begin{Lemma}
\label{basic}
Let $G$ be a connected graph.
Then, %we have the following:
\begin{itemize}
\item[(i)]
If $G$ is co-chordal, 
then any cycle of $G$ of length $\geq 5$ has a chord;
\item[(ii)]
If $G$ is $2K_2$-free, 
then any cycle of $G$ of length $\geq 6$ has a chord.
\end{itemize}
\end{Lemma}

The toric ideals $I_G$ of $2K_2$-free graphs $G$
are studied in \cite{indispensable}.
(In \cite{indispensable}, $2K_2$-free graphs are called in a different name.)
On the other hand, 
the edge ideals $I(G)$ of $2K_2$-free graphs $G$
are studied by many researchers.
See, e.g., \cite{Nevo} and \cite{NePe}
together with their references and comments.
(In these papers,  $2K_2$-free graphs are called ``{\em $C_4$-free graphs}.")
One can characterize the toric ideals $I_{\widehat{G}}$ of $\widehat{G}$
that are generated by quadratic binomials
in terms of  $2K_2$-free graphs.

\begin{Theorem}
\label{suspension}
Let $G$ be a finite connected graph.
Then the following conditions are equivalent:
\begin{itemize}
\item[(i)]
$I_{\widehat{G}}$ is generated by quadratic binomials;
\item[(ii)]
$G$ is $2K_2$-free and
$I_G$ is generated by quadratic binomials;
\item[(iii)]
$G$ is $2K_2$-free and
satisfies the condition (i) in Proposition \ref{quadcriterion}.
\end{itemize}
\end{Theorem}

\begin{proof}
((i) $\Rightarrow $ (ii))
Suppose that $I_{\widehat{G}}$ is generated by quadratic binomials.
Then $\widehat{G}$ satisfies the conditions in Proposition \ref{quadcriterion}.
Since $G$ is an induced subgraph of $\widehat{G}$,
$I_G$ is generated by quadratic binomials by Proposition \ref{inducedsubgraph}.
Assume that $G$ is not $2K_2$-free.
Then, the graph in Figure \ref{ribbon} is an induced subgraph of $\widehat{G}$.
This contradicts that $\widehat{G}$ satisfies the condition 
(ii) in Proposition \ref{quadcriterion}.

\begin{figure}%[h]
  \begin{center}
    \scalebox{0.5}{\includegraphics{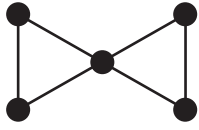}}
  \end{center}
  \caption{Two triangles having one common vertex.}
  \label{ribbon}
\end{figure}

((ii) $\Rightarrow $ (i))
Suppose that $G$ satisfies condition (ii)
and that $I_{\widehat{G}}$ is not generated by quadratic binomials.
Then the graph $\widehat{G}$ does not satisfy one of the conditions
in Proposition \ref{quadcriterion}.
Note that, since $G$ satisfies the conditions in Proposition \ref{quadcriterion},
if an even cycle or two odd cycles do not satisfy the conditions,
then they have the vertex $n+1$.

If an even cycle $C$ of length $\geq 6$ has the vertex $n+1$,
then any other vertices of $C$ are incident with $n+1$.
Thus $C$ has an even-chord.
If minimal odd cycles $C_1$ and $C_2$ have no common vertex and
$C_1$ contains $n+1$,
then $n+1$ is incident with all vertices of $C_2$.
Thus, $C_1$ and $C_2$ has at least three bridges.
Finally, suppose that minimal odd cycles $C_1$ and $C_2$
have exactly one common vertex $v$
and that $C_1$ contains $n+1$.
If $v \neq n+1$, then let $s$ $ (\neq v)$ be a vertex of $C_2$.
Then, since we have
$\{n+1, s\} \in E(\widehat{G}) \setminus (E(C_1) \cup E(C_2)) $,
$C_1$ and $C_2$ satisfy the condition (ii) in 
Proposition \ref{quadcriterion}.
Let $v=n+1$.
Since $C_1$ and $C_2$ are minimal and have the vertex $n+1$,
the length of $C_1$ and $C_2$ is 3 and hence,
$C_1 \cup C_2$ is the graph in Figure \ref{ribbon}.
If $C_1$ and $C_2$ do not satisfy
 the condition (ii) in 
Proposition \ref{quadcriterion},
then $C_1 \cup C_2$ is an induced subgraph of $\widehat{G}$.
Thus, $2K_2$ is an induced subgraph of $G$. 
This is a contradiction.

((ii) $\Rightarrow $ (iii))
It follows from Proposition \ref{quadcriterion}.

((iii) $\Rightarrow $ (ii))
Suppose that 
$G$ satisfies the condition (i) in Proposition \ref{quadcriterion}
and
$G$ is $2K_2$-free.
It is enough to show that $G$ satisfies the 
conditions (ii) and (iii) in Proposition \ref{quadcriterion}.
Let $C_1$ and $C_2$ be minimal odd cycles having exactly one common vertex $v$.
Then there exist edges $\{i,j\} \in E(C_1)$ and $\{k,\ell\} \in E(C_2)$
such that $v \notin \{i,j,k,\ell\}$.
Since $G$ is $2K_2$-free, one of 
$
\{i,k\}, 
\{i,\ell\}, 
\{j,k\}, 
\{j,\ell\}
$ belongs to $E(G)$.
Thus $C_1$ and $C_2$ satisfy the condition (ii) in  Proposition \ref{quadcriterion}.
Let $C_1$ and $C_2$ be minimal odd cycles having no common vertex.
Since $G$ is $2K_2$-free, 
for each edges $\{i,j\} \in E(C_1)$ and $\{k,\ell\} \in E(C_2)$,
one of 
$
\{i,k\}, 
\{i,\ell\}, 
\{j,k\}, 
\{j,\ell\}
$ belongs to $E(G)$.
It then follows that there exist at least two bridges between $C_1$ and $C_2$.
\end{proof}

\begin{Example}
In general, 
there is no implication between
the two conditions
% of a graph $G$:
%\begin{itemize}
%\item
(1) $I_G$ is generated by quadratic binomials and
%\item
(2) $G$ is $2K_2$-free.
%\end{itemize}
In fact,
\begin{itemize}
\item[(a)]
Let $G$ be the graph in Figure \ref{odd}.
\begin{figure}%[h]
  \begin{center}
    \scalebox{0.5}{\includegraphics{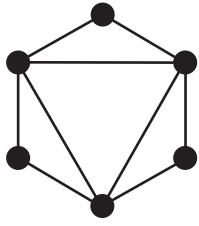}}
  \end{center}
  \caption{An even cycle with three odd chords.}
  \label{odd}
\end{figure} 
Then, $I_G$ is not generated by quadratic binomials.
On the other hand, $G$ is co-chordal (and hence $2K_2$-free).
\item[(b)]
If $G$ is a bipartite graph consisting of 
a cycle $C$ of length 6 and a chord of $C$,
then $I_G$ is generated by two quadratic binomials.
On the other hand, $G$ is not $2K_2$-free.
\end{itemize}
Thus, both $(1) \Rightarrow (2)$ and $(2) \Rightarrow (1)$
are false.
\end{Example}

By using the theory of the Rees ring of edge ideals,
we have a necessary condition for 
$I_{\widehat{G}}$ to possess  a quadratic Gr\"obner basis.

\begin{Proposition}
\label{HHX}
Let $G$ be a connected graph.
If  $I_{\widehat{G}}$ possesses a quadratic Gr\"obner basis,
then  $G$ is co-chordal.
\end{Proposition}

\begin{proof}
Suppose that  $I_{\widehat{G}}$ possesses a quadratic Gr\"obner basis.
Then, by \cite[Corollary 10.1.8]{JuTa}, 
each power of the edge ideal $I(G)$ of $G$ has 
a linear resolution.
Hence, in particular, $I(G)$ itself has a linear resolution.
By Fr\"oberg's theorem \cite[Theorem 9.2.3]{JuTa},
$G$ is co-chordal as desired.
\end{proof}

The converse of Proposition \ref{HHX}
is false in general.
See, e.g., Example \ref{ex:another-minimal}.
However, if $G$ is bipartite, then these conditions are equivalent:

\begin{Theorem}
\label{bipartite}
Let $G$ be a bipartite graph.
Then the following conditions are equivalent:
\begin{itemize}
\item[(i)]
$I_{\widehat{G}}$ is generated by quadratic binomials;
\item[(ii)]
$K[\widehat{G}]$ is Koszul;
\item[(iii)]
$I_{\widehat{G}}$ possesses a quadratic Gr\"obner basis;
\item[(iv)]
$G$ is $2K_2$-free;
\item[(v)]
$G$ is co-chordal.
\end{itemize}
\end{Theorem}

\begin{proof}
First, (v) $\Rightarrow $ (iv) is trivial.
By Proposition \ref{HHX}, we have (iii) $\Rightarrow $ (v).

((iv) $\Rightarrow $ (i))
Suppose that $G$ is $2K_2$-free.
By Lemma \ref{basic}
%,
%any even cycle of $G$ of length $\geq 6$ has a chord.
%Thus, by 
and Proposition \ref{Koszulbipartite}, $I_G$ is generated by quadratic binomials.
Hence (i) follows from Theorem \ref{suspension}.

((i) $\Leftrightarrow $ (ii) $\Leftrightarrow $ (iii))
Since $G$ is bipartite, any odd cycle of $\widehat{G}$
has the vertex $n+1$.
Then by \cite[Proposition 5.5]{centrally},
there exists a bipartite graph $G'$ such that
$I_{\widehat{G}} = I_{G'}$.
By Proposition \ref{Koszulbipartite},
$I_{G'}$ is generated by quadratic binomials if and only if
$I_{G'}$ possesses a quadratic Gr\"obner basis.
Thus, three conditions (i), (ii) and (iii) are equivalent as desired.
\end{proof}

\begin{Remark}
Bipartite graphs satisfying one of the conditions in Theorem \ref{bipartite}
are called {\em Ferrers graphs} (by relabeling the vertices).
The edge ideal $I(G)$ of a Ferrers graph $G$ is well-studied.
See, e.g., \cite{CoNa} and \cite{DoEn}.
\end{Remark}

If $G$ is not bipartite, then the conditions (i) and (ii)
in Theorem \ref{bipartite} are not equivalent.
In fact, 
%the following example is given in \cite[Example 2.1]{quadratic}:

\begin{Example}
\label{wheel}
Let $G$ be a cycle of length 5.
Then $\overline{G}$ is also a cycle of length 5.
Hence $G$ is not co-chordal but $2K_2$-free.
%Since $I_G = (0)$, 
By Theorem \ref{suspension} and Proposition \ref{HHX},
$I_{\widehat{G}}$ is generated by quadratic binomials and
has no quadratic Gr\"obner bases.
Note that $\widehat{G}$ is the graph in Example \ref{syarin}
and that $K[\widehat{G}]$ is not Koszul.
\end{Example}

%Example \ref{wheel} is ``minimal" among such examples
%in the following sense.
Recall that a finite connected simple graph $G$ is called
{\em $(*)$-minimal} if $G$ satisfies the condition
\begin{enumerate}
\item[$(*)$]
$I_{G}$ is generated by quadratic binomials and $I_{G}$ 
possesses no quadratic Gr\"obner basis
\end{enumerate}
and if no induced subgraph
$H$ $(\neq G)$ satisfies the condition $(*)$.
The suspension graph $\widehat{G}$ given in Example \ref{wheel} is
a $(*)$-minimal graph.
We generalize this example and give a nontrivial infinite series of 
 $(*)$-minimal graphs:

\begin{Theorem}
\label{-cycle}
Let $G$ be the graph on the vertex set $[n]$
whose complement is a cycle of length $n$.
If $n \geq 5$, then $\widehat{G}$ is  $(*)$-minimal,
i.e., $\widehat{G}$ satisfies the following:
\begin{itemize}
\item[(i)]
$I_{\widehat{G}}$ is generated by quadratic binomials;
\item[(ii)]
$I_{\widehat{G}}$ has no quadratic Gr\"obner basis;
\item[(iii)]
For any induced subgraph $H$ ($\neq \widehat{G}$) of 
$\widehat{G}$, the toric ideal
$I_H$ of $H$ possesses a quadratic Gr\"obner basis.
\end{itemize}
\end{Theorem}

\begin{proof}
Since a cycle of length $n \geq 5$ is not chordal,
(ii) follows from Proposition \ref{HHX}.

Next, we will show (iii).
In \cite[Theorem 9.1]{Stu},
a quadratic Gr\"obner basis ${\mathcal G}_n$ of the toric ideal
of the complete graph $K_n$ of $n$ vertices is constructed.
In the proof, the vertices of $K_n$ are identified with 
the vertices of a regular $n$-gon in the plane labeled clockwise from $1$ to $n$. 
The Gr\"obner basis ${\mathcal G}_n$ consists of quadratic binomials $f$ such that
%\begin{itemize}
%\item
%$f$ corresponds to an even cycle of length 4;
%\item
%where 
the initial monomial of $f$
corresponds to 
a pair of non-intersecting edges of $K_n$
and the non-initial monomial of $f$
corresponds to 
a pair of intersecting edges of $K_n$.
%\end{itemize}
Note that
the edges $\{1,2\},\{2, 3\}, \ldots, \{n-1,n\},\{1,n\}$
do not appear in the non-initial monomial in each binomial of ${\mathcal G}_n$.

Let $H$ be an induced subgraph of $\widehat{G}$.
If $H = G$, then $\overline{G}$ is the cycle $C = 
(\{1,2\},\{2,3\},\ldots,\{n-1, n\},\{1,n\} )$.
By the above observation on ${\mathcal G}_n$,
we have a quadratic Gr\"obner basis of $I_G$ by the elimination
of ${\mathcal G}_n$.
If $H \neq G$, then $\overline{H}$ is a graph all of whose connected components
are paths.
Since $\overline{H}$ is a subgraph of the cycle $C = 
( \{1,2\},\{2,3\},\ldots,\{n-1, n\},\{1,n\} )$,
we have a quadratic Gr\"obner basis of $I_H$ by the elimination
of ${\mathcal G}_n$.

Finally, we will prove the condition (ii).
By the condition (iii), $I_G$ is generated by quadratic binomials.
Moreover, since $\overline{G}$ is the cycle of length $n \geq 5$,
$G$ is $2K_2$-free.
Thus, we have (i) by Theorem \ref{suspension} as desired.
\end{proof}

Even if $G$ is co-chordal, $\widehat{G}$ may be 
$(*)$-minimal:

\begin{Example} 
\label{ex:another-minimal}
Let $G$ be the graph whose complement is
the chordal graph in Figure \ref{odd}.
Then, $I_{\widehat{G}}$ is generated by quadratic binomials
since $G$ is co-chordal (and hence $2K_2$-free) and $I_G =(0)$.
On the other hand,
%$K[\widehat{G}]$ is not Koszul
%and 
%hence $I_{\widehat{G}}$ has no quadratic Gr\"obner bases.
computational experiments in Section 3 show that
$\widehat{G}$ is $(*)$-minimal.
\end{Example}

\section{Computational experiments}
In this section, we enumerate all finite connected simple graphs $G$
satisfying the condition $(*)$ up to 8 vertices
by utilizing various software.
Proposition \ref{quadcriterion} is a key of 
our enumeration method.

Proposition \ref{quadcriterion} gives an algorithm
to determine if a toric ideal $I_G$ is generated by quadratic binomials.
Since the criteria in Proposition \ref{quadcriterion} are
characterized by cycles of $G$,
we need to enumerate all even cycles and minimal odd cycles of $G$
in order to implement the algorithm.
We implement the algorithm by utilizing
CyPath \cite{cypath} which is a cycles and paths enumeration program
implemented by T.~Uno.
The algorithm is used at step (\ref{item:criterion}) of
the following procedure to search for the graphs satisfying $(*)$.

\begin{enumerate}
\item (generating step) \label{item:generating} \\
We use nauty \cite{nauty} as a generator
of all connected simple graphs with $n$ vertices
up to isomorphism.

\item (criterion step) \label{item:criterion} \\
The criteria in Proposition \ref{quadcriterion} detect graphs $G$
whose toric ideals $I_G$ are generated by quadratic binomials.
These are candidates for satisfying the condition $(*)$.

\item (exclusion step) \label{item:exclusion} \\
For each candidate $G$, we iterate the following computation.
\begin{enumerate}
\item \label{item:vector}
A new weight vector $w$ is chosen randomly on each iteration.
\item We compute a Gr\"obner basis of the toric ideal $I_G$
with respect to the chosen weight vector $w$ with Risa/Asir \cite{asir}.
\item If the Gr\"obner basis is quadratic
then the graph $G$ is excluded from candidates.
\end{enumerate}

\item (final check step) \label{item:check} \\
We check the Koszul property of $K[G]$ with Macaulay2 \cite{M2}.
If it is not Koszul then $I_G$ possesses no quadratic Gr\"obner basis.
If it is indeterminable then we compute all Gr\"obner bases by using
TiGERS \cite{tigers} or CaTS \cite{cats}.
\end{enumerate}

In our experimentation, we take 10000 to be the number of iterations
at step (\ref{item:exclusion}) in the case of $8$ vertices.
Then, there are 214 graphs as remaining candidates and
we can check that 213 graphs of these are not Koszul with Macaulay2.
The last one is indeterminable by computational methods in our environment.
However, Theorem \ref{-cycle} tells us
that it has no quadratic Gr\"obner basis,
because it is the suspension of the complement graph
of a cycle whose length is $7$.
Therefore, we complete classification of the finite graphs with $8$ vertices.  
Table \ref{table:number} shows numbers of
(\ref{item:generating}) the connected simple graphs,
(\ref{item:criterion}) the graphs whose toric ideals $I_G$ are
generated by quadratic binomials (include number of zero ideals),
(\ref{item:check}) the graphs satisfying $(*)$
(include number of the graphs which have degree $1$ vertices)
respectively.

\begin{table}[htbp]
\begin{center}
\begin{tabular}{|c||r|rr|rr|} \hline
vertices  & \makebox[4em][c]{(1)} & \multicolumn{2}{|c|}{(2)} & \multicolumn{2}{|c|}{(4)} \\ \hline
3         &     2 &    2 &  (2) &    0 &      \\ \hline
4         &     6 &    6 &  (3) &    0 &      \\ \hline
5         &    21 &   20 &  (7) &    0 &      \\ \hline
6         &   112 &   95 & (14) &    1 &  (0) \\ \hline
7         &   853 &  568 & (34) &   14 &  (2) \\ \hline
8         & 11117 & 4578 & (78) &  214 & (51) \\ \hline
\end{tabular}
\end{center}
\caption{} \label{table:number}
\end{table}

We list the $14$ graphs (Figures \ref{fig:7-318}--\ref{fig:7-820})
satisfying $(*)$ with $7$ vertices.
Figure \ref{fig:7-815} belongs to the infinite series in Theorem \ref{-cycle}
and Figure \ref{fig:7-358} is the $(*)$-minimal graph
in Example \ref{ex:another-minimal}.
The list for the graphs with $8$ vertices is
available at

\bigskip

\begin{center}
URL: {\tt http://www2.rikkyo.ac.jp/\~{}ohsugi/minimalexamples/}
\end{center}

\begin{figure}[!htbp]
\begin{minipage}{0.32\textwidth}
\begin{center}
\includegraphics[width=5.5cm]{\figseven/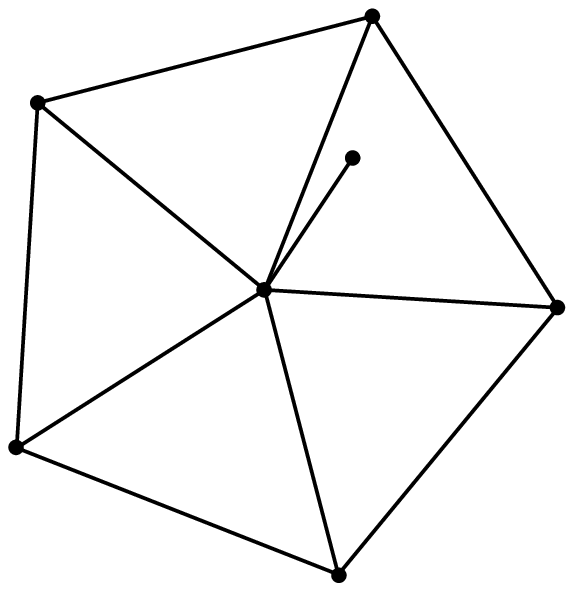}
\end{center}
\vspace{-9mm}
\caption{} \label{fig:7-318}
\end{minipage}
\begin{minipage}{0.32\textwidth}
\begin{center}
\includegraphics[width=5.5cm]{\figseven/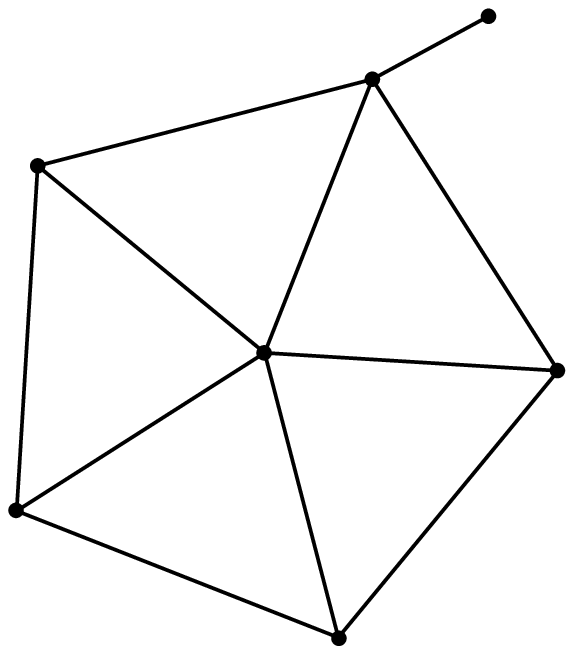}
\end{center}
\vspace{-9mm}
\caption{} \label{fig:7-345}
\end{minipage}
\begin{minipage}{0.32\textwidth}
\begin{center}
\includegraphics[width=5.5cm]{\figseven/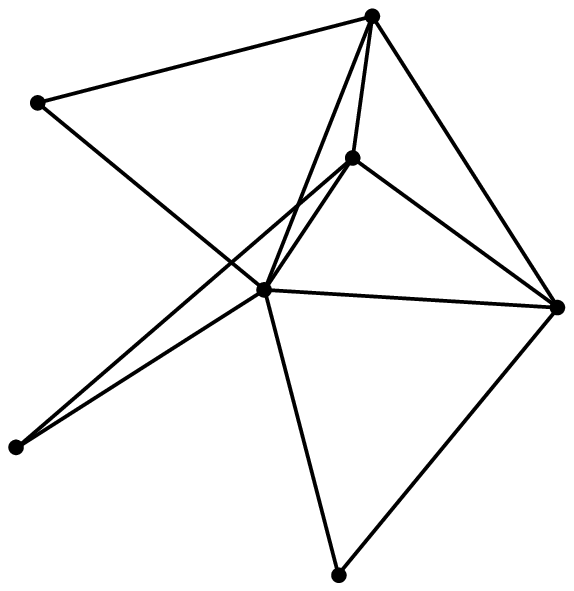}
\end{center}
\vspace{-9mm}
\caption{} \label{fig:7-358}
\end{minipage}
\\
\begin{minipage}{0.32\textwidth}
\begin{center}
\includegraphics[width=5.5cm]{\figseven/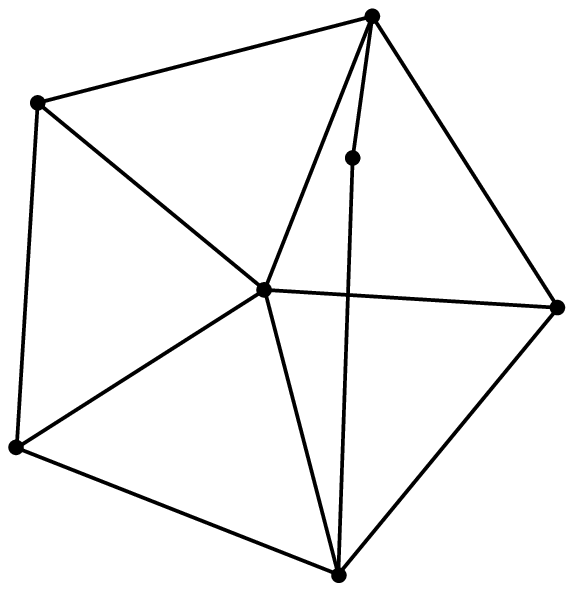}
\end{center}
\vspace{-9mm}
%\vspace{-1cm}
\caption{} \label{fig:7-521}
\end{minipage}
\begin{minipage}{0.32\textwidth}
\begin{center}
\includegraphics[width=5.5cm]{\figseven/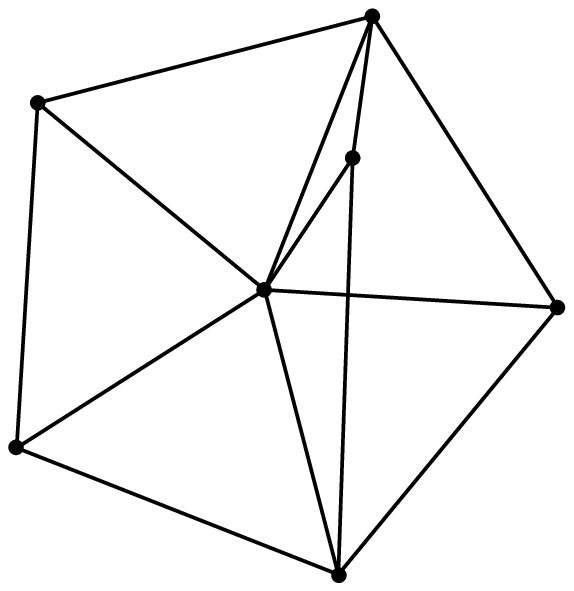}
\end{center}
\vspace{-9mm}
%\vspace{-1cm}
\caption{} \label{fig:7-522}
\end{minipage}
\begin{minipage}{0.32\textwidth}
\begin{center}
\includegraphics[width=5.5cm]{\figseven/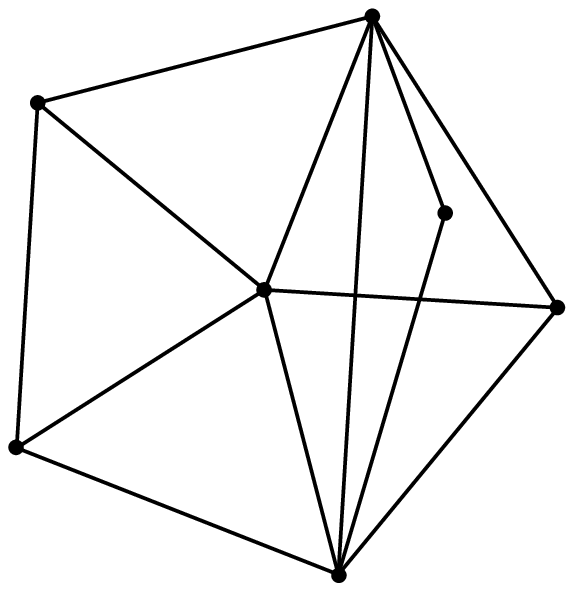}
\end{center}
\vspace{-9mm}
%\vspace{-1cm}
\caption{} \label{fig:7-538}
\end{minipage}
\\
\begin{minipage}{0.32\textwidth}
\begin{center}
\includegraphics[width=5.5cm]{\figseven/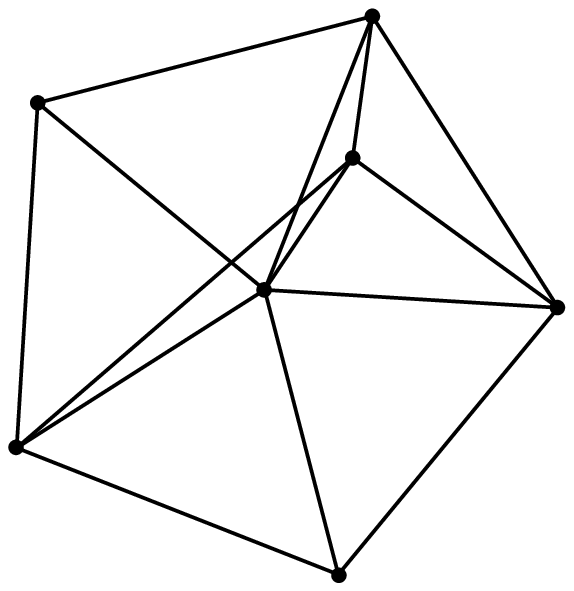}
\end{center}
\vspace{-9mm}
%\vspace{-1cm}
\caption{} \label{fig:7-701}
\end{minipage}
\begin{minipage}{0.32\textwidth}
\begin{center}
\includegraphics[width=5.5cm]{\figseven/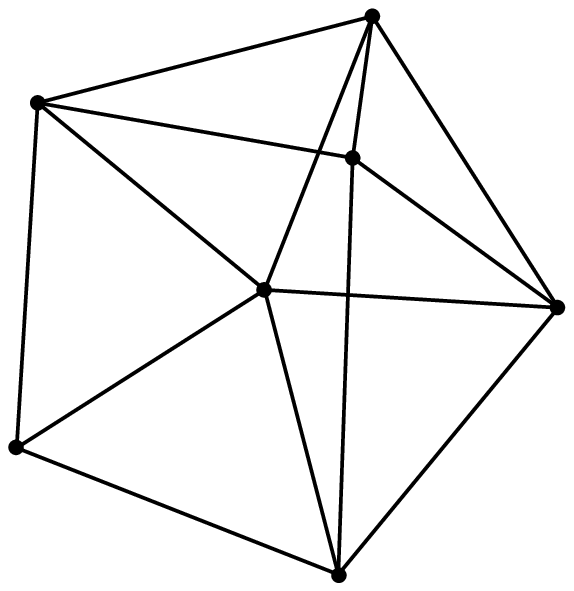}
\end{center}
\vspace{-9mm}
%\vspace{-1cm}
\caption{} \label{fig:7-804}
\end{minipage}
\begin{minipage}{0.32\textwidth}
\begin{center}
\includegraphics[width=5.5cm]{\figseven/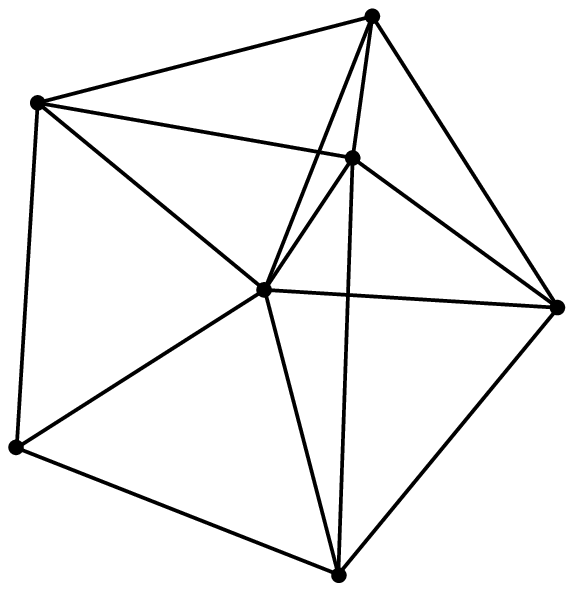}
\end{center}
\vspace{-9mm}
%\vspace{-1cm}
\caption{} \label{fig:7-808}
\end{minipage}
\\
\begin{minipage}{0.32\textwidth}
\begin{center}
\includegraphics[width=5.5cm]{\figseven/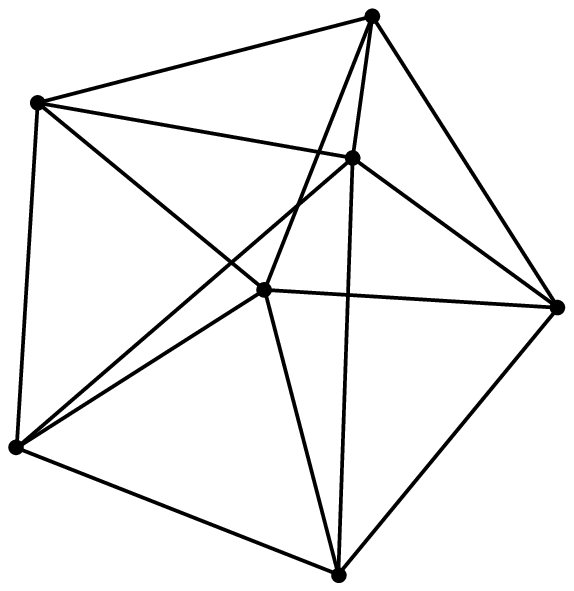}
\end{center}
\vspace{-9mm}
%\vspace{-1cm}
\caption{} \label{fig:7-809}
\end{minipage}
\begin{minipage}{0.32\textwidth}
\begin{center}
\includegraphics[width=5.5cm]{\figseven/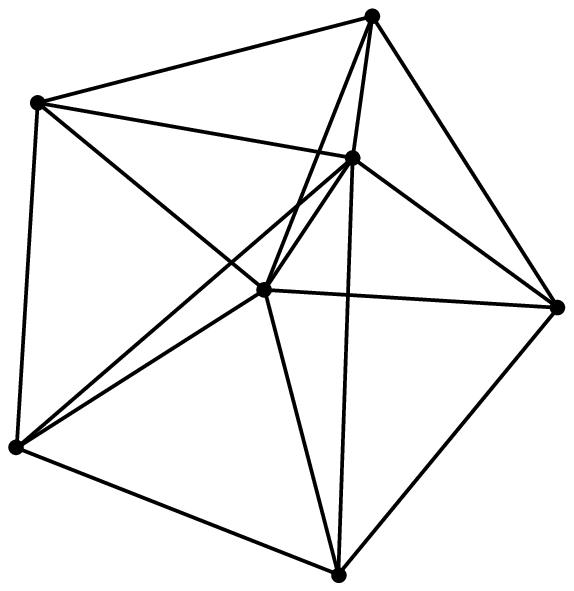}
\end{center}
\vspace{-9mm}
%\vspace{-1cm}
\caption{} \label{fig:7-810}
\end{minipage}
\begin{minipage}{0.32\textwidth}
\begin{center}
\includegraphics[width=5.5cm]{\figseven/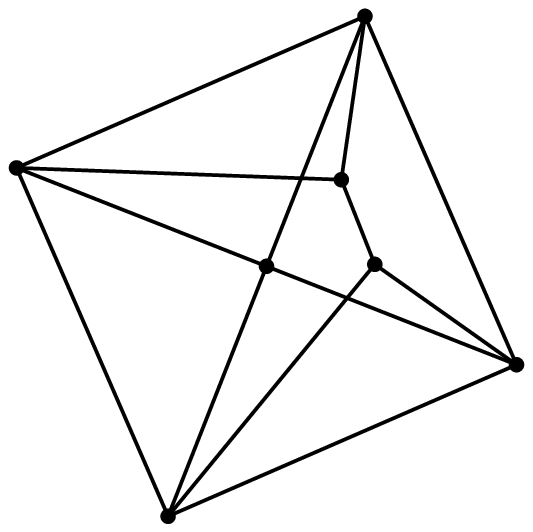}
\end{center}
\vspace{-9mm}
%\vspace{-1cm}
\caption{} \label{fig:7-813}
\end{minipage}
\\
\begin{minipage}{0.32\textwidth}
\begin{center}
\includegraphics[width=5.5cm]{\figseven/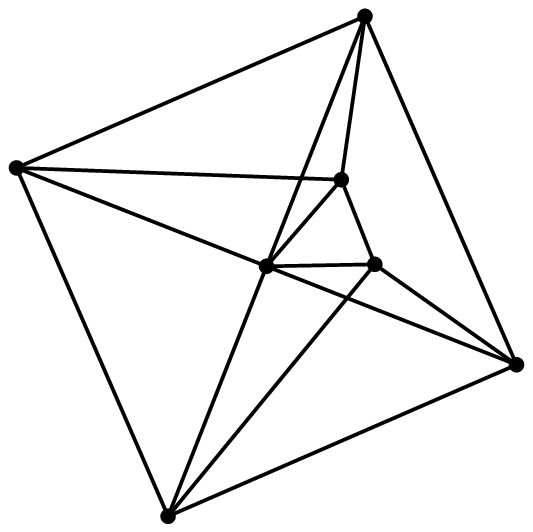}
\end{center}
\vspace{-9mm}
%\vspace{-1cm}
\caption{} \label{fig:7-815}
\end{minipage}
\begin{minipage}{0.32\textwidth}
\begin{center}
\includegraphics[width=5.5cm]{\figseven/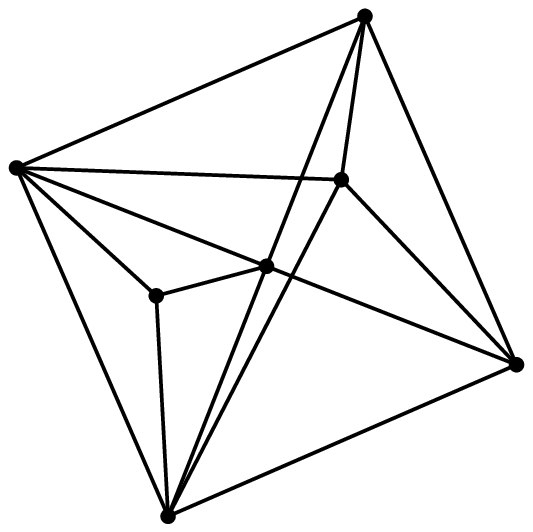}
\end{center}
\vspace{-9mm}
%\vspace{-1cm}
\caption{} \label{fig:7-820}
\end{minipage}
\begin{minipage}{0.32\textwidth}
\begin{center}
\quad
\end{center}
\end{minipage}
\end{figure}

\bigskip

\section*{Acknowledgement}
This research was supported by the JST (Japan Science and Technology Agency)
CREST (Core Research for Evolutional Science and Technology)
research project
{\em Harmony of Gr\"obner Bases and the Modern Industrial Society}
in the frame of the JST Mathematics Program
``Alliance for Breakthrough between Mathematics and Sciences.''

\end{document}